\documentclass[12pt]{amsart}
\usepackage{amsthm,amssymb,amsmath,enumerate}
\usepackage{tikz,pgflibraryarrows}
\usepackage{colordvi,color}
\usepackage[pdftex]{hyperref}
\theoremstyle{plain}
\newtheorem{thm}{Theorem}[section]
\newtheorem{cor}[thm]{Corollary}
\newtheorem{lem}[thm]{Lemma}
\newtheorem{prop}[thm]{Proposition}

\theoremstyle{definition}
\newtheorem{defi}[thm]{Definition}
\newtheorem{defis}[thm]{Definitions}
\newtheorem{exam}[thm]{Example}
\newtheorem{exams}[thm]{Examples}

\theoremstyle{remark}
\newtheorem{rem}[thm]{Remark}

\def\clsp{\overline{\operatorname{span}}}

\def\on{\operatorname}

\title[Group actions]{Group actions on labeled graphs and their $C^*$-algebras}
\date{\today}

\author[Teresa Bates]{Teresa Bates}
\address{Teresa Bates \\
School of Mathematics and Statistics, \\
The University of NSW \\
UNSW  Sydney 2052 \\ Australia} 
\email{teresa@unsw.edu.au}

\author[David Pask]{David Pask}
\address{David Pask, School of Mathematics and Applied Statistics, \\
University of Wollongong, \\ NSW 2522, Australia}
\email{dpask@uow.edu.au}

\author[Paulette Willis]{Paulette Willis}
\address{Paulette Willis, Department of Mathematics
 651 PGH, \\
University of Houston, \\
Houston, TX 77204-3008}
\email{pnwillis@math.uh.edu}

\begin{document}

\numberwithin{equation}{section}

\begin{abstract}
We introduce the notion of the action of a group on a labeled
graph and the quotient object, also a labeled graph. We define
a skew product labeled graph and use it to prove
a version of the Gross-Tucker theorem for labeled graphs.
We then apply these results to the $C^*$-algebra associated to
a labeled graph and provide some applications in nonabelian
duality.
\end{abstract}



\thanks{The second author was supported by the Australian Research Council.  The third author was supported by the NSF Mathematical Sciences Postdoctoral Fellowship DMS-1004675, the University of Iowa Graduate College Fellowship as part of the Sloan Foundation Graduate Scholarship Program, and the University of Iowa Department of Mathematics NSF VIGRE grant DMS-0602242.}

\maketitle

\section{Introduction}

A labeled graph $( E , \mathcal{L} )$ is a directed graph $E=(E^0,E^1,r,s)$ together with
a function $\mathcal{L} : E^1 \to \mathcal{A}$ where $\mathcal{A}$ is called the alphabet.
Labeled graphs are a model for studying symbolic dynamical systems;
the labeled path space is a shift space whose properties may be inferred
from the labeled graph presentation (cf.\ \cite{lm}). Labeled graph algebras were introduced in \cite{bp,bp2}, their theory has been developed in \cite{bcp,jk,jkp} and has found applications in mirror quantum spheres in \cite{rs}.

The main purpose of this paper is to  introduce the notion of a group action on a labeled graph and  study the crossed products formed by the induced action on the associated $C^*$-algebra. Before we do this we update the definition of the $C^*$-algebra associated to a labeled graph. 
In order to circumvent a technical error in the literature we add a new condition to ensure that the resulting $C^*$-algebra satisfies a version of
the gauge-invariant uniqueness Theorem.
Since a directed graph is a labeled graph where $\mathcal{L}$ is injective, we will
be generalizing a suite of results for directed graphs and their $C^*$-algebras (see \cite{dpr,kp,kqr}).
This is not as straightforward as it may seem since two distinct edges may carry
the same label,  so new techniques will be needed to prove our results.

An action of a group $G$ on a labeled graph $(E, \mathcal{L} )$ is an action of $G$
on $E$ together with a compatible action of $G$ on $\mathcal{A}$ so that we may sensibly define the quotient object $(E/G, \mathcal{L}/G)$ as a labeled graph.
In \cite{gt} Gross and Tucker introduce the notion of a skew product graph $E \times_c G$ formed from
a map $c : E^1 \to G$ and show that $G$ acts freely on $E \times_c G$ with quotient $E$.
The Gross-Tucker Theorem~\cite[Theorem 2.1.2]{gt} takes a free action of $G$ on $E$ and recovers (up to equivariant isomorphism)
the original graph and action from the quotient graph $E/G$. One might speculate
that a similar result holds for free actions on labeled graphs. In section \ref{sec:group_actions}
we describe a skew product construction for labeled graphs and prove a version of
the Gross-Tucker theorem for free actions on labeled graphs (Theorem~\ref{gttlg}). Since a group action
on a labeled graph is a pair of compatible actions, a new approach is needed: In Definition~\ref{splg}
we define a skew product labeled graph $( E \times_c G , \mathcal{L}_d )$ to be a skew-product graph
$E \times_c G$ together with a labeling $\mathcal{L}_d : ( E \times_c G )^1 \to \mathcal{A} \times G$ which is defined using a new function $d : E^1 \to G$.
The purpose of the new function $d$ is to accommodate the possibility that two edges carry the same label.
In Remark~\ref{explaind} we discuss the importance of $d$.

We then turn our attention to applications of our  results on labeled graph actions to the $C^*$-algebras, $C^* ( E , \mathcal{L} )$ we have associated to labeled graphs.

A function $c : E^1 \to G$ on a directed graph gives rise to a coaction $\delta$ of $G$ on $C^* (E)$
such that $C^* (E) \times_\delta G \cong C^* ( E \times_c G)$ (cf.\
\cite{kqr}). In Proposition~\ref{coact} we show that a skew product labeled
graph $( E \times_c G , \mathcal{L}_d )$ gives rise to a coaction $\delta$ of $G$
on $C^* ( E , \mathcal{L} )$ provided that $c : E^1 \to G$ is
consistent with the labeling map $\mathcal{L} : E^1 \to \mathcal{A}$. Then in Theorem~\ref{thm:skewiscoact}
we show that $C^* (E , \mathcal{L} ) \times_\delta G \cong C^* ( E \times_c G , \mathcal{L}_\mathbf{1} )$ where  $\mathbf{1}: E^1 \to G$ is given by $\mathbf{1} (e) = 1_G$ for all $e \in E^1$. Since this isomorphism is equivariant for the dual action of $G$ on $C^* (E , \mathcal{L} ) \times_\delta G$ and the action of $G$ on $C^* ( E \times_c G , \mathcal{L}_\mathbf{1} )$ induced by left translation of $G$ on $( E \times_c G , \mathcal{L}_\mathbf{1} )$, Takai duality then gives us
\[
C^* ( E \times_c G , \mathcal{L}_{\mathbf{1}} ) \times_{\tau,r} G \cong C^* ( E , \mathcal{L} ) \otimes \mathcal{K} ( \ell^2 ( G ) )
\]

\noindent in Corollary~\ref{lciso}. Indeed if $d$ is consistent with the labeling map $\mathcal{L} : E^1 \to \mathcal{A}$, then
$C^* ( E \times_c G , \mathcal{L}_d )$ is equivariantly isomorphic to $C^* ( E \times_c G , \mathcal{L}_\mathbf{1} )$ (see Proposition~\ref{prop:onewilldo}). 

For a directed graph $E$ a function $c : E^1 \to \mathbb{Z}$ given by $c(e)=1$ for all $e \in E^1$ gives rise to a skew product graph $E \times_c G$ whose $C^*$-algebra which is strongly Morita equivalent to the fixed point algebra $C^* (E)^\gamma$ for the gauge action. In the case of labelled graphs if $c,d : E^1 \to \mathbb{Z}$ are given by $c(e) = 1$, $d(e)=0$ for all $e \in E^1$, then $C^* ( E \times_c G , \mathcal{L}_d )$ is strongly Morita equivalent to $C^* ( E , \mathcal{L} )^\gamma$ (see Theorem~\ref{thm:fpa}).

An action $\alpha$ of $G$ on a directed graph $E$ induces an action of $G$ on $C^* (E)$, moreover if the action
is free, then using the Gross--Tucker Theorem we have
\begin{equation} \label{eq:kp}
C^* (E) \times_{\alpha,r} G \cong C^* ( E/G ) \otimes \mathcal{K} ( \ell^2 ( G ) )
\end{equation}

\noindent by \cite[Corollary 3.10]{kp}. In Theorem~\ref{thm:induceact} we show that an action of $G$ on
$(E , \mathcal{L} )$ induces an action of $G$ on $C^* (E , \mathcal{L} )$.
If we wish to use the Gross-Tucker Theorem for labeled graphs to prove the labeled graph analog \eqref{eq:kp} we
need to know when the maps $c,d : (E/G)^1 \to G$ provided by Theorem~\ref{gttlg} are consistent with the quotient labeling
$\mathcal{L}/G$. The answer to this question is provided by Theorem~\ref{gttlc}: It happens precisely when the action $\alpha$ has
a fundamental domain. Hence, if the free action of $G$ on $( E , \mathcal{L} )$ has a fundamental domain, then in  Corollary~\ref{isomorphisms} we show that
\[
C^* ( E , \mathcal{L} ) \times_{\alpha,r} G \cong C^* ( E / G  , \mathcal{L} / G ) \otimes \mathcal{K} ( \ell^2 ( G ) ) .
\]

\section{Labeled Graphs and their $C^*$-algebras}

We begin with a collection of definitions, which are taken from \cite{bp}.
A \emph{directed graph} $E=(E^0, E^1, r, s)$ consists of a vertex set $E^0$, an edge set $E^1$,
and range and source maps $r,s: E^1 \to E^0$. We shall assume  throughout this paper that $E$ is row-finite and essential, that is
\[
 r^{-1}(v)\neq \emptyset \text{ and }  1 \le \# s^{-1} (v) < \infty
\]

\noindent for all $v \in E^0$. We let $E^n$ denote the set of paths of length $n$ and set $E^+ = \cup_{n \ge 1} E^n$.

\begin{defi}
A \emph{labeled graph} $(E, \mathcal{L})$ over an alphabet $\mathcal{A}$ consists of a
directed graph $E$ together with a labeling map $\mathcal{L} : E^1 \rightarrow \mathcal{A}$.
\end{defi}

\noindent
We may assume that $\mathcal{L}: E^1 \to \mathcal{A}$ is surjective.  Let $\mathcal{A}^*$ be the collection of all \emph{words} in the symbols of $\mathcal{A}$.  For $n \ge 1$ the map $\mathcal{L}$ extends naturally to a map $\mathcal{L}: E^n \rightarrow \mathcal{A}^*$: for $\lambda = \lambda_1 \cdots \lambda_n \in E^n$ we set $\mathcal{L} ( \lambda ) = \mathcal{L}( \lambda_1 ) \cdots \mathcal{L} ( \lambda_n )$ and we say that $\lambda$ is a \emph{representative} of the labeled path $\mathcal{L} ( \lambda )$.  Let $\mathcal{L}( E^n )$ denote the collection of all labeled paths in $( E, \mathcal{L} )$ of length $n$. Then $\mathcal{L}^+(E)= \cup_{n \geq 1} \mathcal{L} ( E^n )$ denotes the collection of all labeled paths in $(E, \mathcal{L})$, that is all words in the alphabet $\mathcal{A}$ which may be represented by paths in $E$.

\begin{exams} \label{trivex}
  \begin{enumerate}[(a)]
    \item \label{triv} Every directed graph $E$ gives rise to a labeled graph $( E , \mathcal{L}_\tau )$ over the alphabet $E^1$ where $\mathcal{L}_{\tau}: E^1 \to E^1$ is the identity map.
    \item \label{fishy}  The directed graph $E$ whose edges $e,f,g$ have been labeled using the alphabet $\{0,1\}$ as shown below is an example of a labeled graph
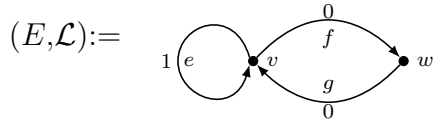
\[
\begin{tikzpicture}
    \def\vertex(#1) at (#2,#3){
        \node[inner sep=0pt, circle, fill=black] (#1) at (#2,#3)
        [draw] {.};
    }
    \node[inner sep=3pt, anchor = south] at (-2.5,0) { {\Large $\scriptstyle (E, \mathcal{L}):=$}};
    \vertex(vertA) at (0,0)
    \node[inner sep=3pt, anchor = west] at (vertA.east)
    {$\scriptstyle v$};
     \vertex(vertB) at (2,0)
    \node[inner sep=3pt, anchor = west] at (vertB.east)
    {$\scriptstyle w$};
    \node[inner sep=2pt, anchor = east] at (-1,0) {$\scriptstyle 1$};
    \node[inner sep=2pt, anchor = west] at (-1,0) {$\scriptstyle e$};
    \draw[style=semithick, -latex] (vertA.north west)
        .. controls (-0.25,0.75) and (-1,0.5) ..
        (-1,0)
        .. controls (-1,-0.5) and (-0.25,-0.75) ..
        (vertA.south west);
    \draw[style=semithick, -latex] (vertA.north east)
    parabola[parabola height=0.5cm] (vertB.north west);
    \node[inner sep=2pt, anchor = south] at (1,0.5) {$\scriptstyle 0$};
    \node[inner sep=2pt, anchor = north] at (1,0.5) {$\scriptstyle f$};
    \draw[style=semithick, -latex] (vertB.south west)
    parabola[parabola height=-0.5cm] (vertA.south east);
    \node[inner sep=2pt, anchor = north] at (1,-0.5) {$\scriptstyle 0$};
    \node[inner sep=2pt, anchor = south] at (1,-0.5) {$\scriptstyle g$};

\end{tikzpicture}
\]
\end{enumerate}
\end{exams}

\noindent Let $( E , \mathcal{L} )$ be a labeled graph. Then for
$\beta \in \mathcal{L}^+ ( E )$ we set
\[
r ( \beta ) = \{ r ( \lambda ) : \mathcal{L} ( \lambda ) = \beta \}  , \quad s ( \beta ) = \{ s ( \lambda ) : \mathcal{L} ( \lambda ) = \beta \} .
\]

\noindent
For $A \subseteq E^0$
and $\beta \in \mathcal{L}^+(E)$ the \emph{relative range of $\beta$
with respect to $A$} is 
\begin{equation*}
r ( A, \beta) = \{r ( \lambda ) : \lambda \in E^+, \mathcal{L} ( \lambda ) =
\beta, s( \lambda ) \in A \}.
\end{equation*}

The labeled graph $(E, \mathcal{L})$ is \emph{left-resolving}, if for all $v \in
E^0$ the map $\mathcal{L}$ restricted to 
$r^{-1}(v)$ is injective. The labeled graph $(E, \mathcal{L})$ is \emph{weakly left-resolving} if for all $A,B \subseteq E^0$ and $\beta \in \mathcal{L}^+ (E)$ we have 
\[ 
r ( A \cap B , \beta ) = r ( A , \beta ) \cap r ( B , \beta ) .
\]

\noindent
If $( E , \mathcal{L} )$ is left-resolving then it is weakly left-resolving.
Examples~\ref{trivex}~\eqref{triv} and ~\eqref{fishy} are examples of
left-resolving labeled graphs. 

A collection $\mathcal{B} \subseteq 2^{E^0}$ of subsets of $E^0$ is \textit{closed under relative ranges for $( E , \mathcal{L} )$} if for all $A \in \mathcal{B}$ and $\beta \in \mathcal{L}^+ (E)$ we have $r ( A , \beta ) \in \mathcal{B}$. If $\mathcal{B}$ is closed under relative ranges for $(E , \mathcal{L} )$, contains $r ( \beta )$ for all $\beta \in \mathcal{L}^+ (E)$ and is also closed under finite intersections and unions, then $\mathcal{B}$ is \textit{accommodating} for $( E , \mathcal {L} )$ and the triple $( E , \mathcal{L} , \mathcal{B})$ is called a \textit{labeled space}. 
Let $\mathcal{E}^{0.-}$ be the smallest accommodating collection  of subsets of $E^0$ for $(E, \mathcal{L})$.

\begin{defi}
For $A \subseteq E^0$ and $n \geq 1$, let $\mathcal{L}^n_A := \{ \beta \in
\mathcal{L} ( E^n ) : A \cap s(\beta) \neq \emptyset \}$ denote those
labeled paths of length $n$ whose source intersects $A$
nontrivially.
\end{defi}

\noindent
Though $E$ is row finite it is possible for $\mathcal{L}^1_{A}$ to be infinite; for example if $\mathcal{L}$ is trivial, then $\mathcal{L}^1_{E^0} = E^1$, which is infinite if $E^1$ is infinite. A labeled space $( E , \mathcal{L} , \mathcal{B} )$ is \textit{set-finite} if $\mathcal{L}^1_A$ is finite for all $A \in \mathcal{B}$. The following definition is given in \cite{bp}.

\begin{defi} \label{lgdef}
A representation of a weakly left-resolving, set-finite labeled space $( E , \mathcal{L} , \mathcal{B}  )$
consists of projections $\{ p_A : A \in \mathcal{B} \}$ and
partial isometries $\{ s_a : a \in \mathcal{A} \}$ such that
\begin{itemize}

\item[(i)] If $A, B \in \mathcal{B}$, then $p_A p_B = p_{A \cap
B}$ and $p_{A \cup B} = p_A + p_B - p_{A \cap B}$, where
$p_\emptyset = 0$.

\item[(ii)] If $a \in \mathcal{A}$ and $A \in
\mathcal{B}$, then $p_A s_a = s_a p_{r ( A, a )}$.

\item[(iii)] If $a , b \in \mathcal{A}$, then $s_a^*
s_a = p_{r( a )}$ and $s_{a}^*s_{b} = 0$ unless $a = b$.

\item[(iv)] For $A \in \mathcal{B}$ we have
\[
p_A = \sum_{a \in \mathcal{L}^1_A} s_{a} p_{r( A , a )} s_{a}^* .
\]
\end{itemize}
\end{defi}

\noindent
$C^* ( E , \mathcal{L}  , \mathcal{B} )$ is the
universal $C^*$-algebra generated by a representation of $( E ,
\mathcal{L} , \mathcal{B}  )$. Let $\gamma : \mathbb{T} \to \operatorname{Aut} C^* ( E , \mathcal{L} , \mathcal{B} )$ be the gauge action determined by
\[
\gamma_z p_A = p_A , \gamma_z s_a = z s_a \text{ for } A \in \mathcal{B},  a \in
\mathcal{A} .
\]

\begin{rem} \label{rem:giut}
The gauge invariant uniqueness Theorem for $C^* ( E , \mathcal{L}  , \mathcal{B} )$ as stated in \cite[Theorem 5.3]{bp} is
incorrect. The authors are grateful to Gow for pointing out the error.
The problem
arises in \cite[Lemma 5.2 (ii)]{bp} as it not possible to prove that the projection $r$
is nonzero under the hypotheses used in \cite{bp}. We are also grateful to Jeong
and Kim for pointing out an  mistake in the formula \cite[Remark 3.5]{bp2} and in \cite[Example 2.4]{jk} which is a direct result of the error discovered by Gow.

The problem in \cite[Lemma 5.2 (ii)]{bp} arises because, under the hypotheses on
a labeled space used in \cite{bp}, it is possible
to have $A \supsetneq B \in \mathcal{B}$ with $p_A = p_B$ in
$C^* ( E , \mathcal{L} , \mathcal{B})$. To rectify this problem we must assume that
$\mathcal{B}$ is closed under relative complements; that is
if $A,B \in \mathcal{B}$ are such that $A \supsetneq B$, then
$A \backslash B \in \mathcal{B}$. 
If $\mathcal{B}$ is closed under relative
complements then we also recover the formula in \cite[Remark 3.5]{bp2}.
\end{rem}

\noindent
Before stating the Gauge Invariant Uniqueness Theorem we give a corrected version of \cite[Lemma 5.2]{bp} using the new hypothesis.

\begin{lem} \label{new}
Let $(E, \mathcal{L} , \mathcal{B} )$ be a weakly left-resolving, set-finite labeled space where $\mathcal{B}$ is closed under relative complements and $\{s_a, p_{A} \}$ be
a representation $(E, \mathcal{L} , \mathcal{B} )$. Let $Y = \{ s_{\alpha_i} p_{A_i}
s_{\beta_i}^* : i = 1 , \ldots , N \}$ be a set of partial
isometries in $C^* ( E, \mathcal{L} , \mathcal{B}  )$ which is
closed under multiplication and taking adjoints. If $q$ is a minimal
projection in $C^* (Y)$ then either
\begin{itemize}
\item[(i)] $q = s_{\alpha_i} p_{A_i} s_{\alpha_i}^*$ for some $1
\le i \le N$

\item[(ii)] $q = s_{\alpha_i} p_{A_i} s_{\alpha_i}^* - q'$ where
$q' = \sum_{l=1}^m s_{\alpha_{k(l)}} p_{A_{k(l)}}
s_{\alpha_{k(l)}}^*$ and $1 \le i \le N$; moreover there is a
nonzero $r = s_{\alpha_i \beta} p_{r ( A_i , \beta ) } s_{\alpha_i
\beta}^* \in C^* ( E, \mathcal{L} , \mathcal{B}  )$ such that $q' r
=0$ and $q \ge r$.
\end{itemize}
\end{lem}

\begin{proof}
By \cite[Lemma 4.4]{bp} any projection in $C^* (Y)$ may be
written as
\[
\sum_{j=1}^n s_{\alpha_{i(j)}} p_{A_{i(j)}} s_{\alpha_{i(j)}}^* -
\sum_{l=1}^m s_{\alpha_{k(l)}} p_{A_{k(l)}} s_{\alpha_{k(l)}}^*
\]

\noindent where the projections in each sum are mutually
orthogonal and for each $l$ there is a unique $j$ such that
$s_{\alpha_{i(j)}} p_{A_{i(j)}} s_{\alpha_{i(j)}}^* \ge
s_{\alpha_{k(l)}} p_{A_{k(l)}} s_{\alpha_{k(l)}}^*$.

If $q = \sum_{j=1}^n s_{\alpha_{i(j)}} p_{A_{i(j)}}
s_{\alpha_{i(j)}}^* - \sum_{l=1}^m s_{\alpha_{k(l)}} p_{A_{k(l)}}
s_{\alpha_{k(l)}}^*$ is a minimal projection in $C^* (Y)$ then we
must have $n=1$. If $m=0$ then $q = s_{\alpha_i} p_{A_i}
s_{\alpha_i}^*$ for some $1 \le i \le N$. If $m \neq 0$ then
\[
q = s_{\alpha_i} p_{A_i} s_{\alpha_i}^* - \sum_{\ell=1}^m s_{\alpha_{k (\ell)}} p_{A_{k(\ell)}} s_{\alpha_{k(\ell)}}^* ,
\]

\noindent
where $A_i , A_{k(\ell)} \in \mathcal{B}$ for $1 \le \ell \le m$.
If we apply Definition~\ref{lgdef}~(iv) we may write 
\[
q = \sum_{j =1}^n s_{\alpha_i \beta_j} p_{r ( A_i , \beta_j )} s_{\alpha_i \beta_j}^* - \sum_{h=1}^t \sum_{\ell=1}^m
s_{\alpha_{k(\ell)} \kappa_h} p_{r(A_{k(\ell)},\kappa_h)} s_{\alpha_{k(\ell)} \kappa_h}^* 
\]

\noindent where all $\alpha_i \beta_j$ and $\alpha_{k(\ell)} \kappa_h$ have the same length.
Since $q$ is a nonzero projection there is $1 \le j \le n$ and  $H_j \subseteq \{ 1 , \ldots , t \} \times \{ 1 , \ldots , m \}$ such that 
$\alpha_i \beta_j = \alpha_{k(\ell)} \kappa_{h}$ for all $(h, \ell ) \in H_j$  and 
\[
Y_j := \bigcup_{(h,\ell) \in H_j} r(A_{k(\ell)},\kappa_h ) \subsetneq r ( A_i , \beta_j ) .
\]

\noindent  
Since $\mathcal{B}$ is closed under finite unions we have $Y_j \in \mathcal{B}$.
Then for this $j$ define 
$X_j = r ( A_i , \beta_j ) \backslash Y_j \neq
\emptyset$, then $X_j \in \mathcal{B}$ since $\mathcal{B}$ is closed under relative complements.
Hence the projection $r = s_{\alpha_i \beta_j} p_{X_j} s_{\alpha_i
\beta_j}^*$  is nonzero and $q \ge r$ since $X_j \subset r ( A_i , \beta_j )$. If we set $q'=s_{\alpha_i} p_{A_i} s_{\alpha_i}^* - q$ then since $X_j \cap Y_j = \emptyset$ we have $q' r = 0$
as required.
\end{proof}

\begin{thm} [Gauge invariant uniqueness Theorem] \label{GIUT}
Let $(E, \mathcal{L} , \mathcal{B} )$ be a weakly left-resolving, set-finite labeled space where $\mathcal{B}$ is closed under relative complements and $\{S_a, P_{A} \}$ be
a representation $(E, \mathcal{L} , \mathcal{B} )$ on Hilbert space.  Take $\pi_{S, P}$ to be the representation of $C^*(E, \mathcal{L} , \mathcal{B})$ satisfying $\pi_{S, P} (s_a)= S_a$ and
$\pi_{S, P} (p_{A})= P_{A}$.  Suppose that $P_{A} \neq 0$
for all $\emptyset \neq A \in \mathcal{B}$ and that there is a strongly continuous
action $\gamma'$ of $\mathbb{T}$ on $C^*(\{S_a, P_{A} \})$ such that for
all $z \in \mathbb{T}$, $\gamma'_z \circ \pi_{S, P}= \pi_{S, P} \circ \gamma_z$.
 Then $\pi_{S, P}$ is faithful.
\end{thm}

\begin{proof}
The proof is the same as given in \cite[Theorem 5.3]{bp}, using Lemma~\ref{new} instead of \cite[Lemma 5.2]{bp}.
\end{proof}

\begin{defi}
Let $( E , \mathcal{L} )$ be a weakly left-resolving, set-finite labeled graph, then we define  $\mathcal{E} ( r , \mathcal{L})$ to be the smallest accommodating collection
of subsets of $E^0$ which is closed under relative complements.
\end{defi}

\begin{rem} \label{rem:erl}
Every $A \in \mathcal{E}^{0,-}$ can be written as $A = \cup_{j=1}^n A_j$ where $A_j =  \cap_{i=1}^{m(j)} r ( \beta_{i}^j )$ and $\beta_{i}^{j} \in \mathcal{L}^+ (E)$ for all $i,j$. Hence, by applications of de Morgan's laws we may show that every $A \in \mathcal{E} ( r , \mathcal{L} )$ can be written in the form $A = \cup_{j=1}^n A_j$ where $A_j =  \cap_{i=1}^{m(j)} r ( \alpha_i^j ) \backslash r ( \beta_{i}^j )$ where $r ( \alpha_i^j ) \supsetneq r ( \beta_i^j)$ and $\alpha_i^j , \beta_{i}^{j} \in \mathcal{L}^+ (E)$ for all $i,j$
\end{rem}

\noindent
This Remark motivates the following definition.

\begin{defi} \label{CKfam}
Let $(E, \mathcal{L})$ be a weakly left-resolving, set-finite labeled graph.  A
\emph{Cuntz-Krieger $(E, \mathcal{L})$-family} consists of  commuting
projections $\{p_{r(\beta)}: \beta \in \mathcal{L}^+(E) \}$ and
partial isometries $\{s_a: a \in \mathcal{A} \}$ with the properties
that:
\begin{enumerate} \renewcommand{\theenumi}{\roman{enumi}}
 \item[(CK1a)] \label{ck1a}  For all $\beta, \omega \in \mathcal{L}^+(E)$, $p_{r(\beta)} p_{r(\omega)} =0$ if and only if $r(\beta) \cap r(\omega) = \emptyset$.
  \item[(CK1b)] \label{ck1b}  For all $\beta, \omega, \kappa \in \mathcal{L}^+(E)$, if $r(\beta) \cap r(\omega) = r(\kappa)$, then
  $p_{r(\beta)}p_{r(\omega)} = p_{r(\kappa)}$, if $r(\beta) \cup r(
  \omega) = r(\kappa)$, then $p_{r(\beta)} + p_{r(\omega)} - p_{r(\beta)}p_{r(\omega)} = p_{r(\kappa)}$ and if $r ( \beta ) \supsetneq r ( \omega )$, then $p_{r(\beta)} - p_{r(\omega)} \neq 0$.
  \item[(CK2)] \label{ck2} If $a \in \mathcal{A}$ and $\beta \in \mathcal{L}^+(E)$, then $p_{r(\beta)} s_a= s_a p_{r(\beta a)}$.
  \item[(CK3)] \label{ck3} If $a,b \in \mathcal{A}$, then $s_a^* s_a= p_{r(a)}$ and $s_a^* s_b=0$ unless $a=b$
  \item[(CK4)] \label{ck4} For $\beta \in \mathcal{L}^+(E)$, if $\mathcal{L}^1_{r(\beta)}$ is finite and non-empty, then we have
    \begin{equation} \label{eq:relation}
    p_{r(\beta)}=\sum_{a \in \mathcal{L}^1_{r(\beta)}}{s_a p_{r(\beta a)}s^*_a}.
    \end{equation}
\end{enumerate}
\end{defi}

\noindent Let $C^* ( E , \mathcal{L}  )$ be the
universal $C^*$-algebra generated by a  Cuntz-Krieger $(E , \mathcal{L} )$-family.

Let $\gamma' : \mathbb{T} \to \operatorname{Aut} C^* ( E , \mathcal{L} )$ be the gauge action determined by
\[
\gamma'_z p_{r ( \beta )} = p_{r (\beta)} , \gamma'_z s_a = z s_a \text{ for } \beta \in \mathcal{L}^+ (E) ,  a \in
\mathcal{A} .
\]

 \begin{thm} \label{thm:links}
Let $(E, \mathcal{L})$ be a weakly left-resolving, set-finite labeled graph. Then $C^*
(E, \mathcal{L})$ is isomorphic  to $C^*(E, \mathcal{L}, \mathcal{E}(r,
\mathcal{L}))$; moreover
\[
C^* ( E , \mathcal{L} ) = \clsp \{ s_\alpha p_A s_\beta^* : \alpha , \beta \in \mathcal{L}^+ (E) , A \in  \mathcal{E} ( r, \mathcal{L} ) \}  .
\]
\end{thm}

\begin{proof}
Let $\{s_a, p_{r(\beta)} \}$ be a universal Cuntz-Krieger $(E,
\mathcal{L})$-family and $\{t_a, q_A\}$ be a universal representation of the
labeled space $(E, \mathcal{L}, \mathcal{E}(r, \mathcal{L}))$.
For $a \in \mathcal{A}$, set $T_a= s_a$.  

By (CK1a) we may define $Q_\emptyset=0$. For $\alpha , \beta \in \mathcal{L}^+ (E)$ we may define $Q_{r ( \alpha ) \cap r ( \beta )} = Q_{r(\alpha)}Q_{r (\beta)}$
and $Q_{r (\alpha \cup r  ( \beta )} = Q_{r (\alpha )} + Q_{r (\beta )} - Q_{r (\alpha ) \cap r (\beta )}$ in $C^* ( E , \mathcal{L})$. If $r ( \alpha ) \supsetneq r ( \beta )$ then we may define $Q_{r ( \alpha ) \backslash r ( \beta )} = Q_{r ( \alpha )} - Q_{r (\beta)} \neq 0$ in $C^* ( E , \mathcal{L})$.By Remark~\ref{rem:erl} and using the inclusion/exclusion law we may define $Q_A$ in $C^* ( E , \mathcal{L})$ for all $A \in \mathcal{E} (r, \mathcal{L})$.

It is a routine calculation to show that $\{T_a, Q_A \}$ is a representation of the labeled space $(E, \mathcal{L}, \mathcal{E}(r, \mathcal{L}))$ in $C^*(E, \mathcal{L})$. By the universal property of $C^*(E, \mathcal{L}, \mathcal{E}(r, \mathcal{L}))$ there exists a homomorphism $\Phi: C^*(E, \mathcal{L}, \mathcal{E}(r, \mathcal{L})) \to C^*(E,
\mathcal{L})$ such that $\Phi(t_a)=T_a \quad \text{and} \quad \Phi(q_A)= Q_A$.
It is straightforward to see that $\gamma'_z \circ \Phi = \Phi \circ \gamma_z$ for $z \in \mathbb{T}$. The first statement then follows by Theorem~\ref{GIUT}, and the final statement follows by applying $\Phi$ to an arbitrary element of $C^*(E, \mathcal{L}, \mathcal{E}(r, \mathcal{L}))$ (see \cite[Lemma 4.4]{bp}).
\end{proof}

\section{Automorphisms of Labeled graphs and their $C^*$-algebras} \label{autolgcalg}

\noindent
We begin by defining what a labeled graph morphism is and use the definition to define a labeled graph automorphism.
Then in Theorem \ref{thm:induceact} we show that a labeled graph automorphism of
$( E , \mathcal{L} )$ induces an automorphism of $C^* ( E , \mathcal{L} )$.

\begin{defi} \label{lgmorph}
Let $(E, \mathcal{L})$ and $(F, \mathcal{M})$ be labeled graphs over alphabets
$\mathcal{A}_E$ and $\mathcal{A}_F$ respectively.  
A \emph{labeled graph morphism} is a triple $\phi := (\phi^0, \phi^1,
\phi^{\mathcal{A}_E}): (E, \mathcal{L}) \to (F, \mathcal{M})$ such that
\begin{enumerate}[(a)]
  \item \label{graphmorph} For all $e \in E^1$ we have $\phi^0 (r(e))= r(\phi^1(e))$ and $\phi^0 (s(e))= s(\phi^1(e))$;
  \item \label{lgmorph2} $\phi^{\mathcal{A}_E}: \mathcal{A}_E \to \mathcal{A}_F$
is a map such that 
$\mathcal{M} \circ \phi^1 = \phi^{\mathcal{A}_E} \circ \mathcal{L}$.
\end{enumerate}

\noindent
If the maps $\phi^0, \phi^1 ,\phi^{\mathcal{A}_E}$ are bijective, then the
triple 
$\phi := (\phi^0, \phi^1, \phi^{\mathcal{A}_E})$ is called a \emph{labeled graph
isomorphism}. In the case that $F=E$, $\mathcal{A}_E= \mathcal{A}_F$ and
$\mathcal{L}=\mathcal{M}$ we call  $( \phi^0 , \phi^1 , \phi^\mathcal{A} )$ a
\emph{labeled graph automorphism}.
\end{defi}

\noindent
For a labeled graph morphism $\phi=(\phi^0, \phi^1, \phi^{\mathcal{A}_E})$
we shall omit the superscripts on $\phi$ when the context in which it is being used is clear. 

The set $ \on{Aut}(E, \mathcal{L})  := \{ \phi : \phi \text{ is a labeled
graph automorphism of }  (E,\mathcal{L} ) \}$ 
forms a group under composition. The following result follows easily from the universal definition of $C^* (E , \mathcal{L} )$.

\begin{thm} \label{thm:induceact}
Let $\phi$ be an automorphism of a weakly left-resolving, set-finite labeled graph 
$(E, \mathcal{L})$ and $\{s_a , p_{r(\beta)} \}$ be a universal Cuntz-Krieger
$(E, \mathcal{L})$-family. The maps $s_a \mapsto s_{\phi(a)}$ and
$p_{r(\beta)} \mapsto p_{\phi (r(\beta))}$ induce an automorphism of $C^*
(E, \mathcal{L})$.
\end{thm}

\section{Skew product labeled graphs and group actions} \label{sec:group_actions}

\noindent
 In this section we shall define a skew product labeled graph and define what it means for a group to act on a labeled graph.

\begin{defi} \label{splg}
Let $(E, \mathcal{L})$ be a labeled graph and let $c,d: E^1 \rightarrow G$ be functions.
The \emph{skew product labeled graph} $(E \times_c G, \mathcal{L}_d)$ over alphabet
$\mathcal{A} \times G$ consists of the skew product graph $( E^0 \times G , E^1
\times G , r_c , s_c )$ where
\[
r_c (e,g) = (r(e), gc(e)) \quad s_c (e,g) = (s(e), g)
\]

\noindent together with the labeling
$\mathcal{L}_d: ( E \times_c G )^1 \to \mathcal{A} \times G$ given by
$\mathcal{L}_d (e,g) := (\mathcal{L} (e), gd(e))$.
\end{defi}

\noindent
Since the labels received by $(v,g) \in (E \times_c G)^0$ are in one-to-one
correspondence with the labels received by $v \in E^0$ it follows that
 if $(E,{\mathcal L})$ is left-resolving, then so is $(E \times_c G,{\mathcal
L}_d)$.

\begin{exam} \label{spexs}
For the labeled graph $(E,\mathcal {L})$ of Examples~\ref{trivex}(\ref{fishy})
let $c,d:  E^1 \to {\mathbb Z}$ be given by
$c(e) = 1$ and $d(e) = 0$ for all $e \in E^1$. Then
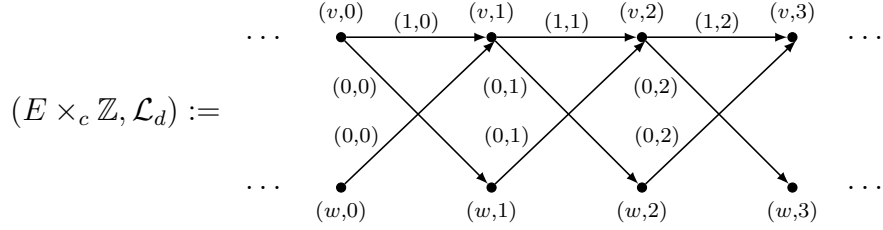
\[
\begin{tikzpicture}
    \def\vertex(#1) at (#2,#3){
        \node[inner sep=0pt, circle, fill=black] (#1) at (#2,#3)
        [draw] {.};   }
        \node at (-3,0){$(E \times_c \mathbb{Z} , \mathcal{L}_d ) :=$};
    \vertex(11) at (0, 1)
    \node[inner sep=3pt, anchor = south] at (11.north)
    {$\scriptstyle (v,0)$};
    \vertex(12) at (0, -1)
    \node[inner sep=3pt, anchor = north] at (12.south)
    {$\scriptstyle (w,0)$};
    \vertex(21) at (2, 1)
    \node[inner sep=3pt, anchor = south] at (21.north)
    {$\scriptstyle (v,1)$};
    \vertex(22) at (2, -1)
    \node[inner sep=3pt, anchor = north] at (22.south)
    {$\scriptstyle (w,1)$};
    \vertex(31) at (4, 1)
    \node[inner sep=3pt, anchor = south] at (31.north)
    {$\scriptstyle (v,2)$};
    \vertex(32) at (4, -1)
    \node[inner sep=3pt, anchor = north] at (32.south)
    {$\scriptstyle (w,2)$};
    \vertex(41) at (6, 1)
    \node[inner sep=3pt, anchor = south] at (41.north)
    {$\scriptstyle (v,3)$};
    \vertex(42) at (6, -1)
    \node[inner sep=3pt, anchor = north] at (42.south)
    {$\scriptstyle (w,3)$};
    \node at (7, 1) {$\dots$};
    \node at (7, -1) {$\dots$};
    \node at (-1, 1) {$\dots$};
    \node at (-1, -1) {$\dots$};

    \draw[style=semithick, -latex] (11.east)--(21.west) node[pos=0.5, anchor=south, inner sep=1pt]{$\scriptstyle (1,0)$};
    \draw[style=semithick, -latex] (11.south east)--(22.north west)node[pos=0.2, circle, anchor=north east, inner sep=0pt]{$\scriptstyle (0,0)$};
    \draw[style=semithick, -latex] (12.north east)--(21.south east) node[pos=0.2, circle, anchor=south east, inner sep=0pt]{$\scriptstyle (0,0)$};

    \draw[style=semithick, -latex] (21.east)--(31.west) node[pos=0.5, anchor=south, inner sep=1pt]{$\scriptstyle (1,1)$};
    \draw[style=semithick, -latex] (21.south east)--(32.north west) node[pos=0.2, circle, anchor=north east, inner sep=0pt]{$\scriptstyle (0,1)$};
    \draw[style=semithick, -latex] (22.north east)--(31.south east) node[pos=0.2, circle, anchor=south east, inner sep=0pt]{$\scriptstyle (0,1)$};

    \draw[style=semithick, -latex] (31.east)--(41.west) node[pos=0.5, anchor=south, inner sep=1pt]{$\scriptstyle (1,2)$};
    \draw[style=semithick, -latex] (31.south east)--(42.north west) node[pos=0.2, circle, anchor=north east, inner sep=0pt]{$\scriptstyle (0,2)$};
    \draw[style=semithick, -latex] (32.north east)--(41.south east) node[pos=0.2, circle, anchor=south east, inner sep=0pt]{$\scriptstyle (0,2)$};

\end{tikzpicture}
\]
\end{exam}

\begin{rem} \label{sppaths}
We shall use the following simpler description of the path space of $E \times_c
G$. For $v \in E^0 , e \in E^1 , g \in G$ set $v_g = (v,g)$, $e_g = (e,g)$. Then
for $\mu \in E^n$ where $n \ge 2$ and $g \in G$ set
\[
\mu_g =  ( \mu_1, g )( \mu_2, g c( \mu_1) ) \cdots ( \mu_n, g c(\mu') ) \in ( E
\times G )^n .
\]

\noindent
For $\mu \in E^*$ the map $( \mu , g) \mapsto \mu_g$ identifies $E^* \times G$
with $( E \times_c G )^*$. Then for $( \mu , g ) \in E^* \times G$ we have
\begin{equation} \label{noters}
s(\mu, g)= (s(\mu), g ) \text{ and } r(\mu, g)= (r(\mu), g c(\mu)).
\end{equation}
\end{rem}

\noindent
Let $(E, \mathcal{L})$ be a labeled graph over the alphabet $\mathcal{A}$.
A \emph{labeled graph action of $G$ on $(E, \mathcal{L})$} is a triple $((E,
\mathcal{L}), G, \phi)$ where
$\phi : G \to \on{Aut}(E, \mathcal{L})$ is a group homomorphism.
In particular, 
for all $e \in E^1$ and $g \in G$ we have
\begin{equation} \label{compatible}
\mathcal{L}(\phi_g(e))=\phi_g(\mathcal{L}(e)).
\end{equation}

If we ignore the label maps, a labeled graph action $((E, \mathcal{L}), G,
\phi)$ restricts to a graph action of $G$ on $E$; we denote this restricted
action by $(E,G,\phi)$. The  labeled graph action $((E, \mathcal{L}), G,
\alpha)$ is \emph{free} if  $\phi_g(v)=v$ for some $v \in E^0$, then $g=1_G$ and if $\phi_g (a)=a$ some $a \in \mathcal{A}$,  then $g=1_G$.

The following lemma shows that skew product labeled graphs
provide a rich source of examples of free labeled graph actions. As the proof is routine, we omit it.

\begin{lem} \label{flga}
Let $(E, \mathcal{L})$ be a labeled graph, $c,d: E^1 \to G$ be functions and $(E \times_c G, \mathcal{L}_d)$
be the associated skew product labeled graph. Then
\begin{enumerate}[(i)]
  \item \label{flga1} For $(x,h) \in (E \times_c G)^i$, $(a,h) \in \mathcal{A} \times G$,  $g \in G$ and $i=0,1$ let $\tau^i_g(x,h)=(x,gh)$ and $\tau^{\mathcal{A}}_g (a,h)= (a,gh)$.  Then $\tau_g= (\tau_g^0, \tau_g^1, \tau_g^{\mathcal{A}})$ is a labeled graph automorphism.
    \item \label{flga3} The map $\tau= (\tau^0, \tau^1, \tau^{\mathcal{A}}): G \to
\on{Aut}(E \times_c G, \mathcal{L}_d)$ defined by  $g \mapsto \tau_g$ is a homomorphism.
 \item \label{flga4} The triple $((E \times_c G, \mathcal{L}_d), G, \tau)$ is a free labeled graph action.
\end{enumerate}
\end{lem}

\begin{defi} \label{flgadef}
The map $\tau=(\tau^0, \tau^1, \tau^{\mathcal{A}}): G \to\operatorname{Aut}(E \times_c G, \mathcal{L}_d)$
as given in Lemma~\ref{flga}~(\ref{flga3}) is called the \emph{left labeled graph translation map},
and the action $((E \times_c G, \mathcal{L}_d), G, \tau)$ the \emph{left labeled graph translation action}.
\end{defi}

\noindent
Two labeled graph actions $((E, \mathcal{L}), G, \phi)$ and $((F,
\mathcal{M}), G, \psi)$ are
\emph{isomorphic} if there is a labeled graph isomorphism $\varphi: (E,
\mathcal{L}) \to (F, \mathcal{M})$ 
which is
\emph{equivariant} in the sense that $ \varphi \circ \phi_g = \psi_g
\circ \varphi$ for all $g \in G$. 

\begin{thm} \label{induceact2}
Let $(E, \mathcal{L})$ be a weakly left-resolving, set-finite labeled graph, and $((E, \mathcal{L}), G, \alpha)$ be a labeled graph action.
Let  $\{s_a ,p_{r(\beta)} \}$ be a universal Cuntz-Krieger $(E, \mathcal{L})$-family.  Then for $h \in G$ the maps
\begin{equation*}
\alpha_h s_a = s_{\alpha_h a} \text{ and } \alpha_h
p_{r(\beta)} = p_{\alpha_h r(\beta)}
\end{equation*}

\noindent
determine an an action of $G$ on $C^*(E, \mathcal{L})$.
If  $((E, \mathcal{L}), G, \phi)$ and $((F, \mathcal{M}), G, \psi)$ are
isomorphic then $C^* ( E , \mathcal{L} ) \times_\phi G \cong C^* ( F , \mathcal{M} ) \times_\psi G$.
\end{thm}

\begin{proof}
Follows by a straightforward application of Theorem~\ref{thm:induceact} and the universal property of crossed products.
\end{proof}

\section{Gross-Tucker Theorem}

\noindent
In this section we prove a version of the Gross-Tucker theorem  for labeled graphs.
For directed graphs, the Gross-Tucker theorem says, roughly speaking, that up to equivariant isomorphism, every free action $\alpha$
of a group $G$ on a directed graph $E$ is a left translation automorphism $\tau$ on a skew product graph $(E/G) \times_c G$ built from the quotient graph $E/G$.
Our aim is to prove a similar result for labeled graphs. The new
ingredient is the map $d: E^1 \to G$ found in the definition of a skew product labeled graph for labeled graphs. Before giving our main result, Theorem \ref{gttlg}, we
introduce some notation.

\begin{defis}
Let $((E, \mathcal{L}), G, \alpha)$ be a labeled graph action.  For $i= 0,1$ and $x \in E^i$ let $Gx:= \{ \alpha^i_g (x): g \in G \}$ and $(E/G)^i= \{Gx : x \in E^i \}$.  For $a \in \mathcal{A}$ let $Ga= \{ \alpha^{\mathcal{A}}_g(a): g \in G \}$ and $\mathcal{A}/G= \{Ga: a \in \mathcal{A}\}$.
\end{defis}

\noindent
The proof of the following lemma is straightforward, so we omit it.

\begin{lem} \label{labelquo}
Let $((E, \mathcal{L}), G, \alpha)$ be a labeled graph action.  The maps $r,s : (E/G)^1 \to (E/G)^0$ given by
\begin{equation} \label{rsmaps}
r(Ge)=Gr(e)\text{ and } s(Ge)=Gs(e) \text{ for } Ge \in (E/G)^1
\end{equation}
and the map $\mathcal{L}/G: (E/G)^1 \to \mathcal{A}/G$ given by $(\mathcal{L}/G)(Ge)=G \mathcal{L}(e)$ are well-defined. Consequently, $(E/G, \mathcal{L}/G)$ is a labeled graph over the alphabet $\mathcal{A}/G$.

 The map $q=(q^0, q^1, q^{\mathcal{A}}): (E, \mathcal{L}) \to (E/G, \mathcal{L}/G)$ given by $q^i(x)=Gx$ for $i=0,1, x \in E^i$ and $q^{\mathcal{A}}(a)= Ga$ for $a \in \mathcal{A}$ is a surjective labeled graph morphism.
\end{lem}

\begin{defi}
Let $((E, \mathcal{L}), G, \alpha)$ be a labeled graph action.   The
\emph{quotient labeled graph} 
$\left(E/G, \mathcal{L}/G \right)$ is the labeled graph described in
Lemma~\ref{labelquo}, the map $q: (E,
\mathcal{L}) \to (E/G, \mathcal{L}/G)$ is the \emph{quotient labeled map}.
\end{defi}

\noindent The following Proposition is an analog of \cite[Theorem 2.2.1]{gt}
whose proof is routine, and so we omit it.

\begin{prop} \label{skewquo}
Let $(E, \mathcal{L})$ be a labeled graph, $c,d: E^1 \to G$ be functions and 
$(E \times_c G, \mathcal{L}_d)$ be the associated skew product labeled graph. 
Let $((E \times_c G, \mathcal{L}_d), G, \tau)$ be the left labeled graph
translation action.  Then 
\[
((E \times_c G) / G, \mathcal{L}_d /G) \cong (E, \mathcal{L}) .
\]
\end{prop}

\begin{exam} \label{ex:freeactex}
Recall the labeled graphs $(E, \mathcal{L})$ and $(E \times_c \mathbb{Z},
\mathcal{L}_d)$ from 
Example~\ref{spexs}.  For the left labeled graph translation action $((E
\times_c \mathbb{Z}, \mathcal{L}_d), \mathbb{Z}, \tau)$ we have $((E \times_c
\mathbb{Z})/ \mathbb{Z}, \mathcal{L}_d/ \mathbb{Z}) \cong (E, \mathcal{L})$ by
Proposition~\ref{skewquo}.
\end{exam}

\noindent The Gross-Tucker theorem is a converse to Proposition~\ref{skewquo}. 
It states that if we have a free action of a group on a labeled graph, then we
can recover the original graph from the quotient via a skew product. Recall the following definition for directed graphs.

\begin{defi}
Let $F, E$ be directed graphs.  A surjective graph morphism $p: F \to E$ has the
\emph{unique path lifting property} if given $u \in F^0$ and $e \in E^1$ with
$s(e) = p^0(u)$ there is a unique edge $f \in F^1$ with $s(f)=u$ and $p^1(f)= e$.
\end{defi}

\begin{rem} \label{ex:quotplp}
 Let $(E,G,\alpha)$ be a free graph action. Then the quotient map $q : E \to E/G$ has the unique path lifting property (see \cite[\S 5]{kp} or \cite[p.67]{gt} for instance).
\end{rem}

\begin{defis}
Let $((E, \mathcal{L}), G, \alpha)$ be a labeled graph action and 
$q=(q^0, q^1, q^{\mathcal{A}}):(E, \mathcal{L}) \to (E/G, \mathcal{L}/G)$ be the
quotient labeled map.  A \emph{section} for $q^i$ is a map $\eta^i: (E/G)^i \to
E^i$ for $i= 0,1$ such that $q^i \circ \eta^i= \on{id}_{(E/G)^i}$.  A
\emph{section} for $q^{\mathcal{A}}$ is $\eta^{\mathcal{A}}: \mathcal{A}/G \to
\mathcal{A}$ such that $q^{\mathcal{A}} \circ \eta^{\mathcal{A}}=
\on{id}_{\mathcal{A}/G}$.
\end{defis}

\begin{lem} \label{eta1}
Let $(E, G, \alpha)$ be a graph action and $q=(q^0, q^1):E \to E/G$ be the
quotient map. 
Given a section $\eta^0$ for $q^0$ there is a unique section $\eta^1$ for $q^1$ 
such that
\begin{equation} \label{eq:eta0eta1}
 s ( \eta^1 ( Ge ) ) =  \eta^0 ( s ( Ge ) ) \text{ for all } e \in E^1 .
\end{equation}
\end{lem}

\begin{proof}
By Remark \ref{ex:quotplp} the quotient map $q : E \to E/G$ has the unique path
lifting property. 
Hence if we fix $Gv \in (E/G)^0$, then for each $Ge \in (E/G)^1$ with $s(Ge)=
Gv$ there is a unique $f \in E^1$ with $q^1 (f)= Ge=Gf$ and $s(f)= \eta^0(Gv)$. 
Put $\eta^1 (Ge)= f$, then $\eta^1 : (E/G)^1 \to E^1$ is well-defined and the
source map on $(E/G)^1$ is well-defined.  Since $q^1 (\eta^1(Ge))=q^1(f)=Ge$ it
follows that $\eta^1$ is a section  satisfying \eqref{eq:eta0eta1}. Uniqueness
of $\eta^1$ follows from the unique path lifting property of $q$.
\end{proof}

\noindent The following is a version of the Gross-Tucker Theorem
(cf.\ \cite[Theorem 2.2.2]{gt}) for labeled graphs.

\begin{thm} \label{gttlg}
Let $((E, \mathcal{L}), G, \alpha)$ be a free labeled graph action.
Let $\eta^0, \eta^{\mathcal{A}}$ be sections for $q^0,
q^{\mathcal{A}}$ respectively.  There are functions $c,
d : (E/G)^1  \to G$ such that $((E, \mathcal{L}), G, \alpha)$
is isomorphic to $((E/G \times_{c} G, (\mathcal{L}/G)_{d}), G, \tau)$.
\end{thm}

\begin{proof}
Fix a section $\eta^0: (E/G)^0 \to E^0$ for $q^0$. By Lemma~\ref{eta1} there is
a section $\eta^1$ for $q^1$satisfying \eqref{eq:eta0eta1}. For $Ge
\in (E/G)^1$ set $f= \eta^1 (Ge)$,  then
\[
q^0(r(\eta^1(Ge)))=  q^0 ( r(f) ) = Gr(f)= r(Gf)= r(Ge)= q^0(\eta^0(r(Ge))).
\]

\noindent
As $(E, G, \alpha)$ is free, there is a unique $h \in G$ such that
$\alpha^0_h \eta^0(r(Ge))= r(\eta^1(Ge))$ and we may set $c(Ge)=h$.
Define $\phi : E/G \times_{c} G
\to E$ by
\[
 \phi^0_{c} (Gv,g) = \alpha^0_g \eta^0(Gv) \text{ and }
\phi^1_{c} (Ge,g)= \alpha^1_g
\eta^1(Ge)
\]

\noindent for $(Gv,g) \in  (E/G \times_{c} G)^0$ and $(Ge, g) \in (E/G
\times_{c} G)^1$.
One checks that $\phi_{c}: (E/G \times_{c} G) \to E$ is an
isomorphism of directed graphs.

We claim that $\phi_c$ is equivariant.
Notice that for all $(Gv,h) \in (E/G \times_{c} G)^0$ and $g \in G$ we
have
\[
\phi^0_{c} (\tau^0_g (Gv,h))= \phi^0_{c} (Gv, gh) = \alpha^0_{gh}
\eta^0 (Gv)=
\alpha^0_g \alpha^0_h \eta^0 (Gv) = \alpha^0_g \phi^0_{c} (Gv,h)
\]

\noindent
 and so $\phi_{c}^0 \circ \tau_g^0 = \alpha_g^0 \circ \phi_{c}^0$ for
all $g \in G$.
The argument for $\phi^1_{c}$ is similar and  our claim follows.

We now construct an equivariant bijection $\phi_{d}^{\mathcal{A}/G
\times G} : \mathcal{A} / G \times G \to \mathcal{A}$ which satisfies condition
\eqref{lgmorph2}
of Definition \ref{lgmorph}.
Fix a section $\eta^{\mathcal{A}}: \mathcal{A}/G \to \mathcal{A}$ for
$q^{\mathcal{A}}$. 
We now define a map $d: (E/G)^1 \to G$. Fix $Ge \in (E/G)^1$ and set $f =
\eta^1 ( Ge)$ so that $q^1 ( f ) = Ge$. Since
\[
 q^{\mathcal{A}} \eta^{\mathcal{A}}(\mathcal{L}/G(Ge))= q^{\mathcal{A}} \eta^{\mathcal{A}}(G \mathcal{L} (f))
= q^{\mathcal{A}} \mathcal{L} \eta^1(Ge)
\]

\noindent and  the graph action $((E, \mathcal{L}), G, \alpha)$ is free, there
is a unique 
$k \in G$ such that $\alpha^{\mathcal{A}}_k \eta^{\mathcal{A}}((\mathcal{L}/G)
(Ge)) = \mathcal{L}(\eta^1(Ge))$ and we may define $d(Ge)=k$.
 The function $d : (E/G)^1 \to G$ described in this way is such that
$d(Ge)$ is the unique element of $G$ with the property that
\begin{equation} \label{chardeta}
\alpha^{\mathcal{A}}_{d(Ge)} \eta^{\mathcal{A}}((\mathcal{L}/G) (Ge)) = \mathcal{L}(\eta^1(Ge)).
\end{equation}

\noindent
For each $(Ga,g) \in \mathcal{A}/G \times G$ we define $\phi_{d}^{\mathcal{A}/G \times G}: \mathcal{A}/G \times G \to \mathcal{A}$ by $\phi^{\mathcal{A}/G \times G}_{d} (Ga,g)= \alpha^{\mathcal{A}}_g \eta^{\mathcal{A}}(Ga)$.
 We claim that $\phi_{d}^{\mathcal{A}/G\times G}$ satisfies 
$\phi^{\mathcal{A}/G \times G}_{d} \circ (\mathcal{L}/G)_{d} =
\mathcal{L} \circ \phi^1_{c}$: By~\eqref{chardeta} for all $(Ge, h) \in
(E/G \times_{c} G)^1$ we have
\[
\phi^{\mathcal{A}/G \times G}_{d} \circ (\mathcal{L}/G)_{d} (Ge,h) = \alpha^{\mathcal{A}}_h \alpha^{\mathcal{A}}_{d(Ge)} \eta_{\mathcal{A}} ({\mathcal{L}}/G (Ge)) = \mathcal{L}(\alpha^1_h \eta^1 (Ge))= \mathcal{L} \circ \phi^1_{c} (Ge,h)
\]

\noindent
as required.

It is straightforward to see that $\phi^{\mathcal{A}/G \times G}_{d}$ is
bijective. To see that $\phi^{\mathcal{A}/G \times G}_{d}$ is equivariant notice that we have
\[
\phi^{\mathcal{A}/G \times G}_{d} (\tau^{\mathcal{A}/G \times G}_g (Ge,
h)) = 
\phi^{\mathcal{A}/G \times G}_{d} (Ge, gh) = \alpha^{\mathcal{A}}_g
\alpha^{\mathcal{A}}_h \eta^{\mathcal{A}} (Ge) = \alpha^{\mathcal{A}}_g
\phi^{\mathcal{A}/G \times G}_{d} (Ge, h)
\]

\noindent
for all $(Ge, h) \in (E/G \times G)^1$ and $g \in G$.  Thus 
$\phi_{{c},{d}}=(\phi^0_{c}, \phi^1_{c},
\phi^{\mathcal{A}/G \times G}_{d})$ is the required labeled
graph isomorphism.
\end{proof}

\begin{rem} \label{explaind}
The possibility that two edges in the quotient graph have the same label 
means that we must choose a separate section $\eta^\mathcal{A}$ for
$q^\mathcal{A}$. In turn means that the function
$d$ given in the definition of a skew product labeled graph plays a crucial role in the reconstruction of
the labeled graph action in Theorem~\ref{gttlg}.
\end{rem}

\begin{exam}  
Recall from Example \ref{ex:freeactex}  the labeled graph $( E \times_c \mathbb{Z} , \mathcal{L}_d )$ has a free action of $\mathbb{Z}$ such
that the quotient labeled graph is $(E , \mathcal{L} )$. We use this example to illustrate the point made in Remark \ref{explaind}:

Suppose we choose a section $\eta^0 : E^0 \to (E \times_c \mathbb{Z} )^0$ such that $\eta^0(v)=(v,0)$ and $\eta^0(w)=(w,2)$, then the section $\eta^1 : E^1 \to ( E \times_c \mathbb{Z} )^1$ as defined in Lemma~\ref{eta1} is  given by $\eta^1(e)= (e,0)$, $\eta^1(f)= (f,0)$, and $\eta^1(g)= (g,2)$  whose
image in $(E \times_c \mathbb{Z} ,\mathcal{L}_d )$ is as shown below.
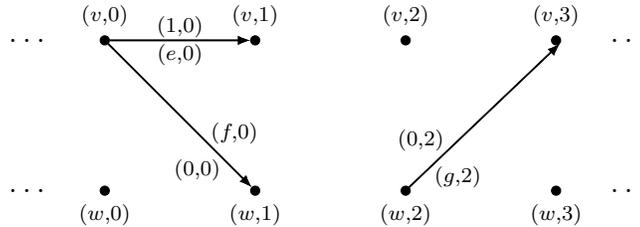
\[
\begin{tikzpicture}
    \def\vertex(#1) at (#2,#3){
        \node[inner sep=0pt, circle, fill=black] (#1) at (#2,#3)
        [draw] {.};   }
    \vertex(11) at (0, 1)
    \node[inner sep=3pt, anchor = south] at (11.north)
    {$\scriptstyle (v,0)$};
    \vertex(12) at (0, -1)
    \node[inner sep=3pt, anchor = north] at (12.south)
    {$\scriptstyle (w,0)$};
    \vertex(21) at (2, 1)
    \node[inner sep=3pt, anchor = south] at (21.north)
    {$\scriptstyle (v,1)$};
    \vertex(22) at (2, -1)
    \node[inner sep=3pt, anchor = north] at (22.south)
    {$\scriptstyle (w,1)$};
    \vertex(31) at (4, 1)
    \node[inner sep=3pt, anchor = south] at (31.north)
    {$\scriptstyle (v,2)$};
    \vertex(32) at (4, -1)
    \node[inner sep=3pt, anchor = north] at (32.south)
    {$\scriptstyle (w,2)$};
    \vertex(41) at (6, 1)
    \node[inner sep=3pt, anchor = south] at (41.north)
    {$\scriptstyle (v,3)$};
    \vertex(42) at (6, -1)
    \node[inner sep=3pt, anchor = north] at (42.south)
    {$\scriptstyle (w,3)$};
    \node at (7, 1) {$\dots$};
    \node at (7, -1) {$\dots$};
    \node at (-1, 1) {$\dots$};
    \node at (-1, -1) {$\dots$};

    \draw[style=semithick, -latex] (11.east)--(21.west) node[pos=0.5, anchor=south, inner sep=1pt]{$\scriptstyle (1,0)$};
     \draw[style=semithick, -latex] (11.east)--(21.west) node[pos=0.5, anchor=north, inner sep=1pt]{$\scriptstyle (e,0)$};
    \draw[style=semithick, -latex] (11.south east)--(22.north west) node[pos=0.75, circle, anchor=north east, inner sep=0pt]{$\scriptstyle (0,0)$};
 \draw[style=semithick, -latex] (11.south east)--(22.north west) node[pos=0.75, circle, anchor=south west, inner sep=0pt]{$\scriptstyle (f,0)$};
    \draw[style=semithick, -latex] (32.north east)--(41.south east) node[pos=0.2, circle, anchor=south east, inner sep=0pt]{$\scriptstyle (0,2)$};
    \draw[style=semithick, -latex] (32.north east)--(41.south east) node[pos=0.2, circle, anchor=north west, inner sep=0pt]{$\scriptstyle (g,2)$};

\end{tikzpicture}
\]
Note that $c(e)=1$, $c(f)=-1$, and $c(g)=3$.

Observe that $f,g \in E^1$ are such that $\mathcal{L} (f) = \mathcal{L} (g)=0$ however,
\[
 \mathcal{L}(\eta^1(f))= \mathcal{L} (f,0) = (0,0) \neq (0,2) = \mathcal{L}
(g,2) = \mathcal{L}(\eta^1(g))
\]

\noindent The function $d$ accounts for this difference.
By Equation~\eqref{chardeta} we have $d (g)=2$, since
$\alpha^{\mathcal{A}}_2(0,0)=(0,2)$, whereas $d(f)=0$.  
Observe that $d(g) \neq d(f)$ even though $\mathcal{L}(g)=
\mathcal{L}(f)$.
\end{exam}

\section{Coactions on Labeled Graph Algebras} \label{coactions}

In \cite{kqr} it is shown that a  function $c : E^1 \to G$ induces a coaction $\delta$ of
$G$ on the graph algebra $C^*(E)$ such that 
$C^* ( E) \times_\delta G \cong C^* (E  \times_c G )$.
One should expect, therefore, that the functions $c,d: E^1 \to G$ would induce a
coaction $\delta$ of $G$ on
$C^* (E , \mathcal{L} )$ such that $C^* ( E , \mathcal{L} ) \times_\delta G
\cong C^* (E \times_c G, \mathcal{L}_d)$.   However in order to obtain such a result we must assume that both functions $c,d$ are label consistent (see Definition~\ref{def:lc} below).  For further
information about coactions of discrete groups see \cite{Q}, amongst others.

\begin{defi} \label{def:lc}
Let $(E, \mathcal{L})$ be a labeled graph over alphabet $\mathcal{A}$.  
A function $c: E^1 \to G$ is \emph{label consistent} if there
is a function $C : \mathcal{A} \to G$ such that $c = C \circ \mathcal{L}$.
\end{defi}

\noindent For any labeled graph $( E , \mathcal{L} )$ the function $\mathbf{1} :
E^1 \to
G$ given by $\mathbf{1} (e) = 1_G$ for all $e \in E^1$ is label consistent.
Firstly we show that if $c$ is label consistent then there is a coaction of $G$ on $C ( E , \mathcal{L} )$.

\begin{prop} \label{coact}
Let  $(E, \mathcal{L})$ be a weakly left-resolving, set-finite labeled graph, $G$ be a discrete group, and $c: E^1
\to G$ be a label consistent function.  Then there is a maximal, normal coaction
$\delta: C^*(E, \mathcal{L}) \to C^*(E, \mathcal{L}) \otimes C^*(G)$ such that
\begin{equation} \label{coaction}
\delta(s_a)= s_a \otimes u_{C(a)} \text{ and }
\delta(p_{r(\beta)})= p_{r(\beta)} \otimes u_{1_G}
\end{equation}

\noindent
where $\{s_a, p_{r(\beta)} \}$ is a universal Cuntz-Krieger $(E,
\mathcal{L})$-family and $\{u_g: g\in G\}$ 
are the canonical generators of $C^*(G)$.
\end{prop}

\begin{proof} The first part of the result follows by the same argument given in \cite[Lemma 3.2]{kqr}. That the coaction $\delta$ is normal and maximal follows by essentially the same arguments as the ones given in \cite[Lemma 3.3]{dpr} and \cite[Theorem 7.1 (v)]{pqr}.
\end{proof}

\noindent
The next result shows that if $d$ is label consistent then we may as well assume that $d = \mathbf{1}$.

\begin{prop} \label{prop:onewilldo}
Let $( E , \mathcal{L} )$ be a weakly left-resolving, set-finite labeled graph and $c : E^1 \to G$
a function. If $d_1 , d_2 : E^1 \to G$ are label consistent functions, then $(  (E
\times_c G , \mathcal{L}_{d_1} ) ,  G , \tau ) \cong ( ( E \times_c G , \mathcal{L}_{d_2} ) ,  G , \tau )$ where $\tau$ is the left translation action.
Hence if $d : E^1 \to G$ is a label consistent function
then there is an isomorphism from $C^* ( E \times_c G , \mathcal{L}_d )$ to $C^* ( E \times_c G ,
\mathcal{L}_{\mathbf{1}} )$ which is equivariant for the $G$--action induced by $\tau$.
\end{prop}

\begin{proof} For the first statement let $\phi^i: ( E \times_c G )^i \to ( E
\times_c G)^i$ be the identity map for $i=0,1$ and
define $\phi^{\mathcal{A} \times G} : \mathcal{A} \times G \to \mathcal{A}
\times G$  by
\[
\phi^{\mathcal{A} \times G} ( a,g ) = ( a , g D_1^{-1} (a) D_2 (a) ) .
\]

\noindent  For $(e,g) \in ( E \times_c G )^1$, after a short calculation  we
have
\[
\phi^{\mathcal{A} \times G} \mathcal{L}_{d_1} (e,g) = ( \mathcal{L} (e) , d_2
(e) ) = \mathcal{L}_{d_2} ( e, g) .
\]

\noindent
It is then straightforward to check that $\phi = ( \phi^0 , \phi^1 ,
\phi^{\mathcal{A} \times G} )$ is a  labeled graph isomorphism. Since for all $h \in G$ we have
\[
\tau_h ( \phi^{\mathcal{A} \times G} ( a,g) ) = ( a , h g D_1^{-1} (a) D_2 (a) ) = 
  \phi^{\mathcal{A} \times G} ( \tau_h ( a,g) )
\]

\noindent it follows that $(  (E \times_c G , \mathcal{L}_{d_1} ) ,  G , \tau ) \cong ( ( E \times_c G , \mathcal{L}_{d_2} ) ,  G , \tau )$.

The final statement follows from Theorem~\ref{induceact2}.
\end{proof}

\begin{rem}
Thanks to Proposition \ref{prop:onewilldo} we may, without loss of generality,
assume that $d = \mathbf{1}$ when we are working with label consistent
$d$-functions. On the other hand it is not hard to see that a different choice
of label consistent functions $c$ will yield non-isomorphic skew-product graphs.
\end{rem}

\noindent
Next we shall show that if $d = \mathbf{1}$ then there is a natural identification $\mathcal{L}_\mathbf{1}^+ ( E \times_c G )$, the labeled path space of $( E \times_c G , \mathcal{L}_\mathbf{1} )$ with $\mathcal{L}^+ (E) \times G$.

\begin{lem} \label{lem:splabels}
Let $( E , \mathcal{L} )$ be a labeled graph and $c: E^1 \to G$ label
consistent. For $\mu \in E^+$ and $g \in G$ the map
\[
\mathcal{L}_\mathbf{1} ( \mu , g ) \mapsto ( \mathcal{L} ( \mu ) , g )
\]

\noindent establishes a bijection from $\mathcal{L}_\mathbf{1}^+ ( E \times_c G )$ to $\mathcal{L}^+ (E) \times G$. 
\end{lem}

\begin{proof}
From Remark~\ref{sppaths} it follows that for $n \ge 1$ every path in $( E \times_c G )^n$ has the form 
$( \mu , g ) = ( \mu_1, g )( \mu_2, g c( \mu_1) ) \cdots ( \mu_n, g c(\mu') )$, for some $\mu \in E^n$ and $g \in G$. Then by definition we have 
\begin{equation} \label{eq:mugdef}
\mathcal{L}_\mathbf{1} ( \mu , g ) = (\mathcal{L}( \mu_1 ), g )(\mathcal{L}(\mu_2), g c(\mu_1) ) \cdots (\mathcal{L}(\mu_n), g c(\mu') ) .
\end{equation}

\noindent If we define the right hand side of \eqref{eq:mugdef} to be $( \mathcal{L} ( \mu ) , g )$ the result follows.
\end{proof}

\noindent
The following Lemma indicates the behavior of the range map under the identification of $\mathcal{L}^+_\mathbf{1} ( E \times_c G )$ with $\mathcal{L}^+ (E) \times G$.

\begin{lem} \label{helpCK}
Let $( E , \mathcal{L} )$ be a labeled graph and $c: E^1 \to G$ be a label consistent function.  Let $a \in \mathcal{A}$, $\beta \in \mathcal{L}^+(E)$, and $g \in G$.  Then under the identification of $\mathcal{L}^+(E) \times G$ with $\mathcal{L}^+_\mathbf{1} (E \times_c G)$ we have $r(\beta,g)= (r(\beta), g C(\beta)) \in\mathcal{E}(r, \mathcal{L}) \times G$.
\end{lem}

\begin{proof}
Observe that for $(\beta,g) \in \mathcal{L}^+(E) \times G$  we have
\begin{equation} \label{rangebetag}\begin{array}{rcl}
r(\beta, g) &=& \{r( \mu , g ) : ( \mu , g ) \in E^* \times G , \mathcal{L} (\mu) = \beta\} \text{ by \eqref{eq:mugdef} } \\
 &=& \{ ( r(\mu), g C (\beta)) : \mathcal{L} (\mu) = \beta \} \text{ by \eqref{noters} }
\end{array}
\end{equation}

\noindent
since the function $c : E^1 \to G$ is label consistent.  Hence we may identify $r(\beta, g)$ with
$(r(\beta), g C(\beta)) \in \mathcal{E} ( r , \mathcal{L} ) \times G$. 
\end{proof}

\noindent
With the above identifications in mind, we turn our attention to the main result of this section.
By Theorem \ref{induceact2} the left labeled graph translation action
$((E \times_c G, \mathcal{L}_\mathbf{1} ), G, \tau)$ defined in
Definition~\ref{flgadef} induces an action
$\tau:G \to \on{Aut} C^*(E \times_c G, \mathcal{L}_\mathbf{1} )$.
When we identify $\mathcal{L}^+_\mathbf{1} ( E \times_c G )$ with $\mathcal{L}^+ (E) \times G$
this action may be described on the generators of $C^*(E \times_c G , \mathcal{L}_\mathbf{1} )$ as follows: For
$h,g \in G$, $a \in \mathcal{A}$, and $\beta \in \mathcal{L}^+(E)$ we have
\begin{equation} \label{inducetau}
\tau_h(s_{(a,g)})= s_{(a, hg)} \text{ and } \tau_h(p_{(r(\beta),g)})=p_{(r(\beta),hg)}.
\end{equation}

\noindent The method of proof for the next result closely follows that of \cite[Theorem
2.4]{kqr}, however we give some of the details as they rely heavily on the identification we  made in Lemma~\ref{helpCK}.

\begin{thm} \label{thm:skewiscoact}
Let  $(E, \mathcal{L})$ be a weakly left-resolving, set-finite labeled graph. Suppose that $G$ is a discrete
group, $c : E^1 \to G$ is a label consistent function, and $\delta$ is the
coaction from Proposition~\ref{coact}.
Let $j_{C^* ( E , \mathcal{L} )} , j_G$ denote the canonical covariant
homomorphisms of $C^* (E , \mathcal{L} )$
and $C^* (G)$ into $M ( C^* ( E , \mathcal{L} ) \times_\delta G )$ and $\{
s_{(a,g)} , p_{(r ( \beta ) ,g)} \}$ be the canonical
generating set of $C^* ( E \times_c G , \mathcal{L}_\mathbf{1} )$. Then the map
$\phi :  C^*(E \times_c G, \mathcal{L}_\mathbf{1} ) \to C^*(E, \mathcal{L})
\times_{\delta} G $
given by
\[
\phi ( s_{(a,g)} ) = j_{C^* ( E , \mathcal{L} )} ( s_a ) j_G ( \chi_{C(a)^{-1}}
) \quad
\phi ( p_{(r(\beta),g)}) = j_{C^* ( E , \mathcal{L} )} ( p_{r ( \beta )} ) j_G (
\chi_{g^{-1}} )
\]

\noindent is an isomorphism.
\end{thm}

\begin{proof}[Sketch of proof]
For each $g \in G$, let $C^*(E, \mathcal{L})_g = \{b \in C^*(E, \mathcal{L}): \delta(b)=b \otimes u_g \}$ denote the corresponding spectral subspace; we write $b_g$ to denote a generic element of $C^*(E, \mathcal{L})_g$.  Then $C^*(E, \mathcal{L}) \times_{\delta} G$ is densely spanned by the set $\{(b_g, h): b_g \in C^*(E, \mathcal{L})_g \text{ and }g,h \in G\}$, and the algebraic operations are given on this set by
\begin{equation*}
(b_g, x)(b_h, y) = (b_g b_h, y) \text{ if } y=h^{-1}x \text{
(and $0$ if not), and } (b_g, x)^*= (b_g^*, gx).
\end{equation*}

\noindent
If $(j_{C^*(E, \mathcal{L})}, j_G)$ denotes the canonical covariant homomorphism of $C^*(E, \mathcal{L})$
into  the multiplier algebra of $C^*(E, \mathcal{L}) \times_{\delta} G$, then $(b_g, x)$ is by definition
$(j_{C^*(E, \mathcal{L})}(b_g) j_G(\chi_{\{x\}}))$.

Using Lemma \ref{helpCK} we may show that  for $(a,g) \in \mathcal{A} \times G$, $\beta \in \mathcal{L}^+ (E)$ and $g \in G$
\begin{equation*}
t_{(a,g)}= (s_a, C(a)^{-1}g^{-1}) \text{ and } q_{(r(\beta),g)}= (p_{r(\beta)}, g^{-1})
\end{equation*}

\noindent is a  Cuntz-Krieger $(E \times_c G, \mathcal{L}_\mathbf{1})$-family in $C^*(E, \mathcal{L}) \times_{\delta} G$.

By universality of $C^*(E \times_c G, \mathcal{L}_\mathbf{1})$ there is a homomorphism $\pi_{t,q}$ from $C^*(E \times_c G, \mathcal{L}_\mathbf{1})$ to $C^*(E, \mathcal{L}) \times_{\delta} G$ such that $\pi_{t,q} (s_{(a,g)}) = t_{(a,g)}$ and $\pi_{t,q} (p_{(r(\beta),g)})= q_{(r(\beta),g)}$ which we may show is injective using the argument from \cite[Theorem 2.4]{kqr} and Theorem~\ref{GIUT}.

Next we show that $\pi_{t,q}$ is surjective.  Observe that $C^*(E, \mathcal{L}) \times_{\delta} G$ is generated by $(s_a, g)$ and $(p_{r(\beta)}, h)$.  Since $\pi_{t,q}(s_{(a, g^{-1} C(a)^{-1})})= t_{(a, g^{-1} C(a)^{-1})}= (s_a, C(a)^{-1} C(a) g)$,
and $\pi_{t,q}(p_{(r(\beta), h^{-1})})= (p_{r(\beta)}, h)$ we see that $\pi_{t,q}$ is surjective.  Hence $\pi_{t,q}$ is the desired isomorphism.

We need to check that $\pi_{t,q}$ is equivariant for the $G$ actions, that is $\pi_{t,q} \circ \tau_g= \widehat{\delta_g} \circ \pi_{t,q}$ for all $g \in G$.  It is enough to check on generators: Notice that for all $s_{(a,h)} \in C^*(E \times_c G, \mathcal{L}_\mathbf{1} )$
\[
\pi_{t,q} \circ \tau_g (s_{(a,h)})= \pi_{t,q}(s_{(a,gh)})=(s_a, C(a)^{-1} h^{-1} g^{-1})= \widehat{\delta}_g (s_a, C(a)^{-1} h^{-1})= \widehat{\delta_g} \circ \pi_{t,q}(s_{(a,h)})
\]

\noindent
and similarly $\pi_{t,q} \circ \tau_g (p_{(r(\beta), h)}) =\widehat{\delta_g} \circ \pi_{t,q}(p_{(r(\beta), h)})$ for $p_{(r(\beta), h)} \in C^*(E \times_c G, \mathcal{L}_\mathbf{1} )$.

We claim that $\pi_{t,q}$ is equivariant for the $\mathbb{T}$ actions, that is  $\pi_{t,q} \circ \gamma_z = (\gamma_z \times G) \circ \pi_{t,q}$ for all $z \in \mathbb{T}$.  It is enough to check this on generators:  Notice that for all $s_{(a,h)} \in C^*(E \times_c G, \mathcal{L}_\mathbf{1} )$ and $z \in \mathbb{T}$ we have
\begin{align*}
\pi_{t,q} \circ \gamma_z (s_{(a,h)})= \pi_{t,q}(z s_{(a,h)})=(z s_a, C(a)^{-1} h^{-1}) &= (\gamma_z \times G)(s_a, C(a)^{-1} h^{-1}) \\
& =(\gamma_z \times_{\delta} G) \circ \pi_{t,q}(s_{(a,h)}) .
\end{align*}

\noindent
Similarly $\pi_{t,q} \circ \gamma_z (p_{(r(\beta), h)})=(\gamma_z \times G) \circ \pi_{t,q}(p_{(r(\beta), h)})$ for all $p_{(r(\beta), h)} \in C^*(E \times_c G, \mathcal{L}_\mathbf{1} )$.  
\end{proof}

\begin{cor} \label{lciso}
Let  $(E, \mathcal{L})$ be a weakly left-resolving, set-finite labeled graph. Suppose that $G$ is a discrete group, $c: E^1 \to
G$ be a label consistent function, and  $\tau$ the induced action of $G$ on
$C^*(E \times_c G, \mathcal{L}_\mathbf{1} )$.  Then
\[
C^*(E \times_c G, \mathcal{L}_\mathbf{1} ) \times_{\tau, r} G \cong C^*(E, \mathcal{L}) \otimes \mathcal{K}(\ell^2(G)).
\]
\end{cor}

\begin{proof}
Since the isomorphism of $C^*(E \times_c G, \mathcal{L}_\mathbf{1})$ with $C^*(E, \mathcal{L}) \times_{\delta} G$ is equivariant for the $G$-actions $\tau, \widehat{\delta}$, respectively, it follows that
\[
C^*(E \times_c G, \mathcal{L}_\mathbf{1}) \times_{\tau, r} G \cong C^*(E, \mathcal{L}) \times_{\delta} G \times_{\widehat{\delta}, r} G.
\]

\noindent
Following the argument in \cite[Corollary 2.5]{kqr}, Katayama's duality theorem \cite{k} gives us that $C^*(E, \mathcal{L}) \times_{\delta} G \times_{\widehat{\delta}, r} G$ is isomorphic to $C^*(E, \mathcal{L}) \otimes \mathcal{K}(\ell^2(G))$,  as required.
\end{proof}

\noindent In order to provide a version of Corollary~\ref{lciso} for group actions we must first
characterise when the functions $c,d$ in the Gross-Tucker Theorem~\ref{gttlg} are label consistent maps. We will do this in the next section.

Recall from \cite[p.209]{Q} that a coaction $\delta$ of a discrete group $G$ on a $C^*$-algebra $A$ is \textit{saturated} if for each $s \in G$ we have $\overline{A_s A_s^*} = A^\delta$ where $A_s$ is the spectral subspace $A_s = \{b \in A : \delta(b)=b \otimes u_s \}$ and $A^\delta$ is  the fixed point algebra for $\delta$ 
\[
A^\delta := \{ b \in A : \delta (a) = a \otimes u_{1_G} \} .
\]

\begin{lem} \label{lem:saturated}
Let $( E , \mathcal{L} )$ be a weakly left-resolving, set-finite labeled graph and $c : E^1 \to \mathbb{Z}$ be given by $c(e) = 1$ for all $e \in E^1$. Then the coaction $\delta$ of $\mathbb{Z}$ on $C^* ( E  , \mathcal{L} )$ induced by $c$ is saturated.
\end{lem}

\begin{proof}
The coaction $\delta$ of $\mathbb{Z}$ on $C^* ( E , \mathcal{L} )$  defined in Proposition~\ref{coact} is such that
the fixed point algebra $C^* (E , \mathcal{L} )^\delta$ is precisely the fixed point algebra $C^* ( E , \mathcal{L} )^\gamma$ for the canonical gauge action of $\mathbb{T}$ on $C^* ( E , \mathcal{L}  )$ by the Fourier transform (cf.\ \cite[Corollary 4.9]{C}. By an argument similar to that in \cite[\S 2]{pr} we have
\[
C^* ( E , \mathcal{L} )^\gamma = \clsp \{ s_\alpha p_A s_\beta^* : \alpha , \beta \in \mathcal{L}^n (E ) ,  A \in \mathcal{E} ( r , \mathcal{L} ) \} .
\]

\noindent Since $E$ has no sinks it follows by a similar argument to that in \cite[Lemma 4.1.1]{pr} that $C^* (E , \mathcal{L} )$ is saturated.
\end{proof}

\begin{thm} \label{thm:fpa}
Let $( E , \mathcal{L} )$ be a weakly left-resolving, set-finite labeled graph. Then $C^* ( E , \mathcal{L} )^\gamma$  is strongly Morita equivalent to $C^* ( E \times_c  \mathbb{Z} , \mathcal{L}_\mathbf{1} )$
where $c : E^1 \to \mathbb{Z}$ is given by $c(e) = 1$ for all $e \in E^1$.
\end{thm}

\begin{proof}
Since $c$ is label consistent it follows by Theorem~\ref{thm:skewiscoact} that
\[
C^* ( E \times_c  \mathbb{Z} , \mathcal{L}_\mathbf{1} ) \cong C^* ( E , \mathcal{L} ) \times_\delta \mathbb{Z}  .
\]

\noindent By Lemma~\ref{lem:saturated} the coaction is $\delta$ is saturated and since
$C^* (E , \mathcal{L} )^\delta \cong C^* ( E , \mathcal{L} )^\gamma$ the result follows.
\end{proof}

\section{Free group actions on labeled graphs} \label{sec:lcg}

\noindent In this section we examine conditions on the free labeled graph
action $((E, \mathcal{L}), G, \alpha)$ which ensure that the functions
$c, d$ from Theorem~\ref{gttlg} are label consistent.

Recall that a fundamental domain for a graph action $(E, G, \alpha)$ is a subset $T$ of $E^0$ such that for every $v \in E^0$ there exists $g \in G$ and a unique $w \in T$ such that $v = \alpha^0_g w$. Every free graph action has a fundamental domain.

\begin{defi} \label{lgfd}
Let $((E, \mathcal{L}), G, \alpha)$ be a free labeled graph action.  A
\emph{fundamental domain for $((E, \mathcal{L}), G, \alpha)$} is a fundamental
domain $T \subseteq E^0$ for  the restricted graph action such that for every $e,f \in E^1$ we have
 \begin{enumerate}[(a)]
   \item  if $r(e), r(f) \in T$ and $G \mathcal{L}(e)= G \mathcal{L}(f)$, then
$\mathcal{L}(e)= \mathcal{L}(f)$ and
   \item  if $s(e), s(f) \in T$ and $G \mathcal{L}(e)= G \mathcal{L}(f)$, then
$\mathcal{L}(e)= \mathcal{L}(f)$.
 \end{enumerate}
\end{defi}

\noindent In Examples~\ref{exams:fd} (i) below we see that not every free action of a group on a labeled graph has a fundamental domain.

\begin{exams} \label{exams:fd}
\begin{enumerate}[(i)]
\item Consider the following labeled graph
 
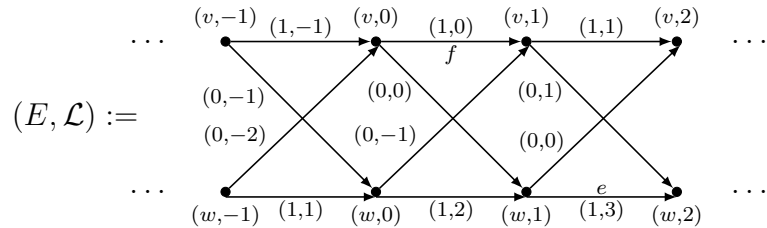
\[\begin{tikzpicture}
    \node at (-2 , 0 ) {$(E , \mathcal{L} ):=$};
     \node[inner sep=0pt, circle, fill=black] (11) at (0, 1) [draw] {.};
    \node[inner sep=3pt, anchor = south] at (11.north) {$\scriptstyle (v,-1)$};
     \node[inner sep=0pt, circle, fill=black] (12) at (0, -1) [draw] {.};
    \node[inner sep=3pt, anchor = north] at (12.south) {$\scriptstyle (w,-1)$};
     \node[inner sep=0pt, circle, fill=black] (21) at (2, 1) [draw] {.};
    \node[inner sep=3pt, anchor = south] at (21.north) {$\scriptstyle (v,0)$};
     \node[inner sep=0pt, circle, fill=black] (22) at (2, -1) [draw] {.};
    \node[inner sep=3pt, anchor = north] at (22.south) {$\scriptstyle (w,0)$};
     \node[inner sep=0pt, circle, fill=black] (31) at (4, 1) [draw] {.};
    \node[inner sep=3pt, anchor = south] at (31.north) {$\scriptstyle (v,1)$};
     \node[inner sep=0pt, circle, fill=black] (32) at (4, -1) [draw] {.};
    \node[inner sep=3pt, anchor = north] at (32.south) {$\scriptstyle (w,1)$};
     \node[inner sep=0pt, circle, fill=black] (41) at (6, 1) [draw] {.};
    \node[inner sep=3pt, anchor = south] at (41.north) {$\scriptstyle (v,2)$};
     \node[inner sep=0pt, circle, fill=black] (42) at (6, -1) [draw] {.};
    \node[inner sep=3pt, anchor = north] at (42.south) {$\scriptstyle (w,2)$};
    \node at (7, 1) {$\dots$};
    \node at (7, -1) {$\dots$};
    \node at (-1, 1) {$\dots$};
    \node at (-1, -1) {$\dots$};
    \draw[style=semithick, -latex] (11.east)--(21.west) node[pos=0.5,
anchor=south, inner sep=1pt]{$\scriptstyle (1,-1)$};
    \draw[style=semithick, -latex] (11.south east)--(22.north west)node[pos=0.2,
circle, anchor=north east, inner sep=0pt]{$\scriptstyle (0,-1)$};
    \draw[style=semithick, -latex] (12.north east)--(21.south east)
node[pos=0.2, circle, anchor=south east, inner sep=0pt]{$\scriptstyle (0,-2)$};
   \draw[style=semithick, -latex] (21.east)--(31.west) node[pos=0.5,
anchor=south, inner sep=1pt]{$\scriptstyle (1,0)$}
   node[pos=0.5, anchor=north, inner sep=1pt]{$\scriptstyle f$};
    \draw[style=semithick, -latex] (21.south east)--(32.north west)
node[pos=0.2, circle, anchor=north east, inner sep=0pt]{$\scriptstyle (0,0)$};
    \draw[style=semithick, -latex] (22.north east)--(31.south east)
node[pos=0.2, circle, anchor=south east, inner sep=0pt]{$\scriptstyle (0,-1)$};
   \draw[style=semithick, -latex] (31.east)--(41.west) node[pos=0.5,
anchor=south, inner sep=1pt]{$\scriptstyle (1,1)$};
    \draw[style=semithick, -latex] (31.south east)--(42.north west)
node[pos=0.2, circle, anchor=north east, inner sep=0pt]{$\scriptstyle (0,1)$};
    \draw[style=semithick, -latex] (32.north east)--(41.south east)
node[pos=0.2, circle, anchor=south east, inner sep=0pt]{$\scriptstyle (0,0)$};
    \draw[style=semithick, -latex] (12.south)--(22.south) node[pos=0.5,
anchor=north, inner sep=1pt]{$\scriptstyle (1,1)$};
     \draw[style=semithick, -latex] (22.south)--(32.south) node[pos=0.5,
anchor=north, inner sep=1pt]{$\scriptstyle (1,2)$};
     \draw[style=semithick, -latex] (32.south)--(42.south) node[pos=0.5,
anchor=north, inner sep=1pt]{$\scriptstyle (1,3)$}
     node[pos=0.5, anchor=south, inner sep=1pt]{$\scriptstyle e$};
\end{tikzpicture}
\]

\noindent
The group $\mathbb{Z}$ acts freely on $(E,\mathcal{L})$ by addition in the
second coordinate of the vertices, edges and labels
as indicated in the picture above; call this action $\alpha$.
Let $T= \{(v,0), (w,1)\}$, then $T$ is a fundamental domain for  the restricted
graph action $(E, \mathbb{Z},\alpha )$.
However when considering the labeled graph action  $((E,\mathcal{L}), \mathbb{Z}
,\alpha)$ the set $T$ does not satisfy Definition~\ref{lgfd} (b). 

Consider the
edges $e,f$ as shown above with $\mathcal{L} (e)=(1,3)$ and
$\mathcal{L}(f)=(1,0)$ respectively. We have $s(e)=(w,1) \in T$ and $s(f)=(v,0)
\in T$ and $\mathbb{Z} \mathcal{L} (e) = \mathbb{Z} \mathcal{L} (f) = \{ (1,n) :
n \in \mathbb{Z} \}$, however $\mathcal{L} (e)= (1,3) \neq (1,0) =
\mathcal{L}(f)$. Indeed any fundamental domain for the restricted action $(E,
\mathbb{Z},\alpha )$ will also fail Definition~\ref{lgfd} (b).

\item Let $c,d:E^1 \to G$ be label consistent functions and $((E \times_c G, \mathcal{L}_d), G, \tau)$ be the associated left labeled graph translation action.  Then one checks that $T= \{(v, 1_G): v \in E^0 \}$ is a fundamental domain for $((E \times_c G, \mathcal{L}_d), G, \tau)$.
\end{enumerate}
\end{exams}

\noindent  The following result shows that when we add the
fundamental domain hypothesis to the free labeled graph action, the
functions $c,d : (E/G)^1 \to G$ in the labeled graph version of the Gross-Tucker
Theorem (Theorem~\ref{gttlg}) may be chosen to be label consistent.

\begin{thm} \label{gttlc}
Let $( (E , \mathcal{L} ) , G , \alpha )$ be a free labeled graph action with a fundamental domain. Then there are label consistent functions $c,d : (E/G)^1 \to G$ such that 
$( ( E , \mathcal{L} ) , G , \alpha ) \cong ( (E/G) \times_c G , ( \mathcal{L} / G)_d ) , G , \tau )$.
\end{thm}

\begin{proof}
Let $T$ be a fundamental domain for $((E, \mathcal{L}), G, \alpha)$. For
every $Gv \in (E/G)^0$ there exists a unique $w \in
T$ such that $Gw= Gv$.  Hence if we define $\eta^0 (Gv)=w$, then $\eta^0:
(E/G)^0 \to T$ is a section for $q^0$. Then we may
define $\eta^1, c,d$, and $\eta^{\mathcal{A}}$  as in
Theorem~\ref{gttlg}.  

It suffices to show that $c$ and $d$ are
label consistent.  To see that $d$ is label consistent suppose $Ge,Gf \in (E/G)^1$ are such that $(\mathcal{L}/G)(Ge)=
(\mathcal{L}/G)(Gf)=Ga \in \mathcal{A}/G$.  Let $b=
\eta^{\mathcal{A}}(Ga) \in \mathcal{A}$, $d (Ge)=k \in G$, and
$d (Gf)= l \in G$.  Then by the definition of $d$ we have
\begin{align}
\mathcal{L}(\eta^1(Ge))&= \alpha^{\mathcal{A}}_k \eta^{\mathcal{A}}
(\mathcal{L}/G)(Ge)= \alpha^{\mathcal{A}}_k b \label{lc1}\\
\mathcal{L}(\eta^1(Gf))&= \alpha^{\mathcal{A}}_l \eta^{\mathcal{A}}
(\mathcal{L}/G)(Gf)= \alpha^{\mathcal{A}}_l b \label{lc2}.
\end{align}

\noindent
This implies that $G \mathcal{L}(\eta^1(Ge))= Ga= G \mathcal{L}(\eta^1(Gf))$
and so $\mathcal{L}(\eta^1(Ge))= \mathcal{L}(\eta^1(Gf))$ since
$s(\eta^1(Ge)), s(\eta^1(Gf)) \in T$.  From Equations~\eqref{lc1} and
~\eqref{lc2} we have $\alpha^{\mathcal{A}}_k b= \alpha^{\mathcal{A}}_l b$ and so
$k=l$ since the $G$ action on $\mathcal{A}$ is free.  Therefore $d$ is
label consistent.

To see that $c$ is label consistent suppose that $Ge,Gf \in (E/G)^1$  are such that $(\mathcal{L}/G)(Ge)=
(\mathcal{L}/G)(Gf)=Ga \in \mathcal{A}/G$, say.  Let $b=
\eta^{\mathcal{A}}(Ga) \in \mathcal{A}$, $c_{\eta}(Ge)=k \in G$, and
$c(Gf)= l \in G$.  Then by the definition of $c$ we have
\begin{align}
r(\eta^1(Ge))&= \alpha^0_k \eta^0 (r(Ge))\label{lcc1}\\
r(\eta^1(Gf))&= \alpha^0_l \eta^0 (r(Gf))\label{lcc2}.
\end{align}

\noindent
Then if we let $e= \alpha^1_{-k}(\eta^1(Ge))$ and $f=
\alpha^1_{-l}(\eta^1(Gf))$ we have $e,f \in E^1$ with $r(e)= \eta^0 (r(Ge)),
r(f)= \eta^0 (r(Gf)) \in T$ and $G \mathcal{L}(e)= G \mathcal{L}(f)$.  Since
$T$ is a fundamental domain we have $\mathcal{L}(e)= \mathcal{L}(f)$ and hence
$\alpha^{\mathcal{A}}_{-k}(\mathcal{L}(\eta^1(Ge)))= \mathcal{L}(e)=
\mathcal{L}(f)= \alpha^{\mathcal{A}}_{-l}(\mathcal{L}(\eta^1(Gf)))$.  Since
$\mathcal{L}(\eta^1(Ge))= \mathcal{L}(\eta^1(Gf))$ we can conclude that
$k=l$ as in the previous paragraph. Therefore $c$ is label consistent 
and our result is established.
\end{proof}

\begin{cor} \label{isomorphisms}
Let $( E , \mathcal{L} )$ be a weakly left-resolving, set-finite labeled graph. Suppose that
$((E, \mathcal{L}), G, \alpha)$ is a free labeled graph action which admits a
fundamental domain.  Then
\[
C^*(E, \mathcal{L}) \times_{\alpha, r} G \cong C^*(E/G, \mathcal{L}/G) \otimes
\mathcal{K}(\ell^2(G)).
\]
\end{cor}

\begin{proof}
By Theorem~\ref{gttlc} there are label consistent functions $c,d: E^1/G \to G$
such that
\[
((E, \mathcal{L}), G, \alpha) \cong ((E/G \times_c G, (\mathcal{L}/G)_d), G,
\tau),
\]
so we have
\[
C^*(E, \mathcal{L}) \times_{\alpha, r} G \cong C^*(E/G \times_c G,
(\mathcal{L}/G)_d) \times_{\tau,r} G .
\]
By Proposition~\ref{prop:onewilldo} and Corollary~\ref{lciso} we have
\begin{align*}
C^*(E/G \times_c G, (\mathcal{L}/G)_d) \times_{\tau,r} G & \cong C^*(E/G \times_c G, (\mathcal{L}/G)_\mathbf{1} ) \times_{\tau,r} G \\
& \cong C^*(E/G,
\mathcal{L}/G) \otimes \mathcal{K}(\ell^2(G))
\end{align*}

\noindent which gives the desired result.
\end{proof}

\end{document}